\documentclass{amsart}
\usepackage{amsmath}
\usepackage{amsfonts}
\usepackage{amssymb}

\newtheorem{theorem}{Theorem}
\theoremstyle{plain}

\newtheorem{claim}{Claim}

\newtheorem{corollary}{Corollary}

\newtheorem{definition}{Definition}

\newtheorem{lemma}{Lemma}

\newtheorem{remark}{Remark}

\numberwithin{equation}{section}

\input{tcilatex}

\begin{document}
\title[]{On the attractor zero of a sequence of polynomials}
\author{hac\`{e}ne belbachir}
\address{USTHB, Faculty of Mathematics Po. Box 32, El Alia, 16111}
\email{hbelbachir@usthb.dz}
\thanks{Thanks for Author One.}
\author{nouar degaichi}
\curraddr{Department of mathematics, Faculty of Exact Sciences, Tebessa
Uiversity, Tebessa 12002, algeria}
\email{ndegaichi@usthb.dz}
\thanks{Thanks for Author Two.}
\date{, 2017}
\dedicatory{}
\thanks{This paper is in final form and no version of it will be submitted
for publication elsewhere.}

\begin{abstract}
The main perpose of this paper is to sudy the roots of a familly of
polynomials that arise from a linear recurrences associated to Pascal's
triangle and their zero attractor, using an analytical methods based on
conformal mappings.
\end{abstract}

\maketitle

\section{ Introduction}

Fibonnaci numbers can be recovered as the sum of the main rays of Pascal's
triangle, that is, each element of Fibonacci sequence $\left( F_{n}\right)
_{n}$ is the sum of binomial coefficients $\binom{n-k}{k}:$ $%
F_{n+1}=\sum\limits_{k}\binom{n-k}{k}$. In this context, one way to extend
the work of Goh et al (cf.\cite{GHR}) is to generalize to the linear
recurrence sequence $\left( T_{n}\right) _{n}$, associated to different
directions of the rays in Pascal's triangle, defined for $n,p,q,r \in 
\mathbb{Z}$  with $n\geq 0,r\geq 1,0\leq p\leq r-1$ and $q+r>0,$ by:

\begin{equation}
T_{n+1}^{(r,q,p)}=\sum\limits_{k=0}^{\left\lfloor \left( n-p\right) /\left(
q+r\right) \right\rfloor }T^{(r,q,p)}\left( n,k\right) =\sum\limits_{k\geq 0}%
\binom{n-qk}{p+rk}x^{n-p-(q+r)k}y^{p+rk},  \tag{1.1}  \label{1.1}
\end{equation}

with the convention $T_{0}=0$.

(Notice that because a sum over empty set is zero)

$(T_{0}=0)T_{1}=$\textperiodcentered \textperiodcentered \textperiodcentered 
$=T_{p}=0$,

and

\begin{equation}
T_{j}=\binom{j-1}{p}x^{j-p-1}y^{p},\ \ \ \ \ \ p+1\leq j\leq r+q+p-1, 
\tag{1.2}  \label{1.2}
\end{equation}

wich is studied by Belbachir et al (cf. \cite{BKS})

$%
\begin{array}{cccccc}
1 &  &  &  &  &  \\ 
1 & 1 &  &  &  &  \\ 
1 & 2 & 1 &  &  &  \\ 
1 & 3 & 3 & 1 &  &  \\ 
1 & 4 & 6 & 4 & 1 &  \\ 
1 & 5 & 10 & 10 & 5 & 1%
\end{array}%
$

\begin{theorem}
The sequence defined in (\ref{1.2}) satisfy the linear recurrence relation:
\end{theorem}

\begin{equation}
\sum\limits_{k=0}^{r}\left( -x\right) ^{k}\binom{r}{k}T_{n-k}=y^{r}T_{n-r-q}
\tag{1.3}  \label{1.3}
\end{equation}

and its generating function (cf. \cite{BKS1})\ is given by:\newline
\begin{eqnarray*}
G(t) &=&\sum\limits_{n\geq 0}T_{n+1}^{\left( r,q,p\right) }t^{n} \\
&=&\frac{y^{p}t^{p+1}(1-xt)^{r-p-1}}{(1-xt)^{r}-y^{r}t^{q+r}}
\end{eqnarray*}

When $q\leq 0,$ the sequence $\left( T_{n}\right) _{n}$ is of order $r$ for
any $q$ $\left( -r<q\leq 0\right) ,$ and the coefficient $y^{r}$ of $%
T_{n-r-q}$ is subtracted from one of coefficients of the terms $%
T_{n-1},...,T_{n-r},$ such that $-r<q\leq 0,$ we get

\begin{equation*}
T_{n}-x\binom{r}{1}T_{n-1}+...+\left( \left( -x\right) ^{r+q}\binom{r}{r+q}%
-y^{r}\right) T_{n-r-q}+...+\left( -x\right) ^{r}\binom{r}{r}T_{n-r}=0
\end{equation*}

Thus, the coefficient of this terms changes status. This is what we call the
Morgan-Voyce phenomenon.

Then the sequence is defined as:

\begin{equation}
T_{n}=2xT_{n-1}-x^{2}T_{n-2}+y^{2}T_{n-3},  \tag{1.4}  \label{1.4}
\end{equation}

We focus our work in a particular case, first $r=2$, $q=1$, $p=1$

In this case we have: $T_{j}=\binom{j-1}{p}x^{j-p-1}y^{p},\ \ \ \ \ \ \ \ \
\ \ \ \ \ \ \ \ \ \ \ \ \ \ \ \ \ \ \ \ \ \ \ \ \ \ \ \ \ \ \ \ \ \ \ \ \ \
\ \ \ \ \ \ 2\leq j\leq 3$

The table below gives us the 10 terms defined in (\ref{1.4})

$T_{0}=0$

$T_{1}=0$

$T_{2}=y$

$T_{3}=2xy$

$T_{4}=3x^{2}y$

$T_{5}=y\left( 4x^{3}+y^{2}\right) $

$T_{6}=xy\left( 5x^{3}+4y^{2}\right) $

$T_{7}=2x^{2}y\left( 3x^{3}+5y^{2}\right) $

$T_{8}=y\left( 7x^{6}+20x^{3}y^{2}+y^{4}\right) $

$T_{9}=xy\left( 8x^{6}+35x^{3}y^{2}+6y^{4}\right) $

$T_{10}=x^{2}y\left( 9x^{6}+56x^{3}y^{2}+21y^{4}\right) $

We know that the sequence defined by (\ref{1.3}) is $r$-periodic

Then by using periodicity we have:

\begin{eqnarray*}
T_{i} & = & yh \left( x^{3},y^{2}\right) \text{ if } i \equiv -1 \mod [ 3 ]
\\
T_{i} & = & xyh\left( x^{3},y^{2}\right) \text{ if } i \equiv 0 \mod [ 3 ] \\
T_{i} & = & x^{2}yh\left( x^{3},y^{2}\right) \text{ if } i \equiv 1 \mod [ 3
]
\end{eqnarray*}

\begin{claim}
The set of roots satisfy 
\begin{equation*}
y^{2}=-x^{3}
\end{equation*}
\end{claim}

The corresponding generating function is:

\begin{eqnarray*}
G(t) &:&=\sum\limits_{n\geq 0}T_{n+1}^{\left( 2,1,1\right) }t^{n} \\
&=&\frac{yt^{2}}{(1-xt)^{2}-y^{2}t^{3}}
\end{eqnarray*}

Then integral representation formula give us:

\begin{lemma}
The polynomial sequence $T_{n}\left( x,y\right) $ given by (\ref{1.4}) has
the integral representation:\newline
For all $x,y$ nonzero real parameters there exist a non negative real number 
$r_{x,y}>0$ such that :%
\begin{equation}
T_{n}(x,y) = \frac{1}{2\pi i} \int_{\vert t\vert =r_{x,y}} \frac{y}{%
(1-xt)^{2}-y^{2}t^{3}} \frac{dt}{t^{n-1}}  \tag{1.5}  \label{1.5}
\end{equation}
\end{lemma}

\begin{proof}
Since $\frac{y}{(1-xt)^{2}-y^{2}t^{3}}\longrightarrow y$ if $%
t\longrightarrow 0$ then we can find $r_{x,y}>0$ such that: $\left\vert
(1-xt)^{2}-y^{2}t^{3}\right\vert \geq 0.9$ for $\left\vert t\right\vert
=r_{x,y}$ then the integral (\ref{1.5}) is well-defined.

Once the integral is well-defined, denoted the integral by $\tilde{T_{n}}%
(x,y) $.

\begin{equation*}
\tilde{T_{n}}\left( x,y\right) =\frac{1}{2\pi i}\int_{ \vert t \vert
=r_{x,y}}\frac{y}{(1-xt)^{2}-y^{2}t^{3}}\frac{dt}{t^{n-1}}.
\end{equation*}

We can directly verify that $\tilde{T_{n}}\left( x,y\right) $ satisfies (\ref%
{1.1}) for $n\geq 1$. Next, since the Taylor expansion of $\frac{y}{%
(1-xt)^{2}-y^{2}t^{3}}=y+2xyt+O\left( t^{2}\right) $, by residue theorem we
have $\tilde{T_{0}}\left( x,y\right) =0$, $\tilde{T_{1}}\left( x,y\right) =0$
and $\tilde{T_{2}}\left( x,y\right) =y$. Hence the initial conditions in (%
\ref{1.1}) are satisfied. thus the integral representation $\tilde{T_{n}}%
\left( x,y\right) $ is a solution to a recursion. since the solution to the
recursion is unique, hence $\tilde{ T_{n}}\left( x,y\right) =T_{n}\left(
x,y\right) .$ this completes the proof of the lemma .
\end{proof}

By a change of variable $y=x\sqrt{x}$ in equation (\ref{1.5}) where $x>0$ we
get:

\begin{eqnarray}
T_{n}\left( x,y\right) & = & \frac{1}{2\pi i}\int_{\vert t \vert =r_{x,y}}%
\frac{y}{(1-xt)^{2}-y^{2}t^{3}}\frac{dt}{t^{n-1}}  \label{1.6} \\
& = & \frac{\left( x\sqrt{x}\right) }{2\pi i}\int_{\vert t\vert =r_{x}}\frac{%
1}{(1-xt)^{2}-x^{3}t^{3}}\frac{dt}{t^{n-1}}  \notag
\end{eqnarray}

On replacing $t$ by $\frac{t}{x}$ in the equation (\ref{1.6}) we obtain:

\begin{eqnarray*}
T_{n}\left( x,y\right) & = &\frac{-x^{n-\frac{1}{2}}}{2\pi i} \int_{\vert t
\vert =r_{x}}\frac{1}{t^{3}-(1-t)^{2}}\frac{ dt}{t^{n-1}} \\
& = &\frac{-x^{n-\frac{1}{2}}}{2\pi i}\int_{\vert t \vert = r_{x}}\frac{1}{%
P(t)}\frac{dt}{t^{n-1}}
\end{eqnarray*}

Where 
\begin{equation}
P(t)=t^{3}-(1-t)^{2}  \tag*{(1.7)}  \label{1.7}
\end{equation}

We put 
\begin{equation}
\tau _{n}(x)= \frac{-1}{2\pi i} \int_{\vert t \vert =r_{x}} \frac{1}{P(t)}%
\frac{dt}{t^{n-1}}  \tag*{(1.8)}  \label{1.8}
\end{equation}

where 
\begin{equation}
\tau _{n}(x)=x^{\left( \frac{1}{2}-n\right) }T_{n}\left( x,y\right) 
\tag{1.9}  \label{1.9}
\end{equation}

\begin{lemma}
1-The polynomial $P(t)$ defined in \ref{1.7} does not have zeros of order 3.%
\newline
2- $P(t)$ has zeros of order 2 if and only if $\left\{ 
\begin{array}{c}
t_{1}=\frac{1-i\sqrt{5}}{3} \\ 
t_{2}=\frac{1+i\sqrt{5}}{3}%
\end{array}%
\right. $
\end{lemma}

\begin{proof}
1- The derivative polynomial $P^{^{\prime }}(t)$ has no zeros of order 2,
consequently $P(t)$ has not zeros of order 3.

2- $P^{^{\prime }}(t)$ has two complex zeros of order 1 $%
\begin{array}{c}
t_{1}=\frac{1-i\sqrt{5}}{3} \\ 
t_{2}=\frac{1+i\sqrt{5}}{3}%
\end{array}%
$\newline
and in this case the zeros of $P(t)$ are: $\frac{1}{3}+i\frac{\sqrt{5}}{3}$, 
$\frac{1}{3}+i\frac{\sqrt{5}}{3}$, $-1-i\frac{\sqrt{5}}{2}$ when $t=t_{2}$,
and when $t=t_{1}$ the zeros of $P(t)$ are: $\frac{1}{3}-i\frac{\sqrt{5}}{3}$%
, $\frac{1}{3}-i\frac{\sqrt{5}}{3}$, $\frac{-1}{2}-i\frac{\sqrt{5}}{5}$.
\end{proof}

Let $t_{1},t_{2},t_{3}$ be the zeros of $P(t)$ arranged via their magnitudes

\begin{equation}
\left\vert t_{1}\right\vert \leq \left\vert t_{2}\right\vert \leq \left\vert
t_{3}\right\vert  \tag{1.10}  \label{1.10}
\end{equation}

$P(t)$ has distinct zeros if $t\neq t_{1}$ and $t\neq t_{2}$ .

After developing the partial fraction decomposition for $\frac{1}{P\left(
t\right) }$ we obtain:

\begin{equation*}
\frac{1}{P\left( t\right) }=\frac{1}{P^{\prime }\left( t_{1}\right) }\frac{1%
}{t-t_{1}}+\frac{1}{P^{\prime }\left( t_{2}\right) }.\frac{1}{t-t_{2}}+\frac{%
1}{P^{\prime }\left( t_{3}\right) }.\frac{1}{t-t_{3}}
\end{equation*}

By using equation \ref{1.8} the integration term by term and by using the
residue theorem we get:%
\begin{equation}
\tau _{n}(x)=\left[ \frac{t_{1}^{1-n}}{P^{\prime }\left( t_{1}\right) }+%
\frac{t_{2}^{1-n}}{P^{\prime }\left( t_{2}\right) }+\frac{t_{3}^{1-n}}{%
P^{\prime }\left( t_{3}\right) }\right]  \tag*{(1.11)}  \label{1.11}
\end{equation}

\begin{remark}
To study the zeros of $\tau _{n}(x)$ it suffices to study the zeros of $P(t)$%
.
\end{remark}

\section{Conformal mappings}

In this section we describe how the zeros of $P(t)$ can be obtained by a
sequence of conformal mappings.

If we set $t=z+\frac{1}{3}$ then $P(t)=0$ return to its canonical form :

\begin{equation*}
z^{3}+\frac{5}{3}z-\frac{11}{27}=0
\end{equation*}

When 
\begin{equation*}
z=\lambda q
\end{equation*}

where 
\begin{equation*}
\lambda =\frac{i\sqrt{5}}{3}
\end{equation*}

Then we have 
\begin{equation*}
q^{3}-3q+\beta =0
\end{equation*}

\bigskip with 
\begin{equation*}
\beta =\frac{-11\sqrt{5}i}{25}
\end{equation*}

\bigskip Next if 
\begin{equation*}
q=p+\frac{1}{p}
\end{equation*}

Then the equation in $p$ is:

\begin{equation*}
p^{6}+\beta p^{3}+1=0
\end{equation*}

Again if 
\begin{equation*}
p^{3}=s
\end{equation*}
then we get: 
\begin{equation*}
s^{2}+\beta s+1=0
\end{equation*}%
which implies 
\begin{equation*}
\beta =-\left( s+\frac{1}{s}\right)
\end{equation*}

Then it is clear how to obtain the zeros of $P(t)$ by going through a
sequence of conformal mappings starting from the $\beta -$plane and
subsequently ending in the $t-$plane .

\begin{definition}
The map : $J(\zeta )=\zeta +\frac{1}{\zeta }$ is called Joukowski map\newline
This map is conformal in the regions $\left\vert \zeta \right\vert <1$ and $%
\left\vert \zeta \right\vert >1$ (cf.\cite{N})
\end{definition}

The $\beta -$plane is mapped into the exterior region to the unit circle in
the $s-$plane under $J^{-1}(-\beta )$. The region is mapped into the
exterior region to the unit circle in the $p-$plane under $p=s^{\frac{1}{3}%
}. $

\begin{remark}
The map $p=s^{\frac{1}{3}}$ is multiple-valued.
\end{remark}

Then we follow $q=J\left( p\right) ,$ $z=\frac{i\sqrt{5}}{3}q$ and $t=z+%
\frac{1}{3}$ to recover the zeros of $P\left( t\right) $ in the $p-$plane
and finally $t$ to $\frac{t}{x}$ in the $x-$plane

We can expressed this situation symbolically as: 
\begin{equation}
\begin{array}{ccc}
\beta \longrightarrow & s\longrightarrow & p_{\downarrow } \\ 
\frac{t}{x}\longleftarrow t\longleftarrow & z\longleftarrow & \longleftarrow
q%
\end{array}
\tag{2.1}  \label{2.1}
\end{equation}

The Joukowski map : $J(\zeta )=\zeta +\frac{1}{\zeta }$ is conformal on $%
\left\vert \zeta \right\vert <1$ and $\left\vert \zeta \right\vert >1$. We
focus the behaviour of $J(\zeta )$ on $\left\vert \zeta \right\vert >1$

Let $\zeta =r\exp i\theta $

then 
\begin{eqnarray*}
J(\zeta ) &=&r\exp \left( i\theta \right) +r^{-1}\exp \left( -i\theta \right)
\\
&=&\left( r+r^{-1}\right) \cos \theta +i\left( r-r^{-1}\right) \sin \theta
\end{eqnarray*}

We set: $\left\{ 
\begin{array}{c}
u=\left( r+r^{-1}\right) \cos \theta \\ 
v=\left( r-r^{-1}\right) \sin \theta%
\end{array}%
\right. $

We obtain: 
\begin{equation*}
\frac{u^{2}}{\left( r+r^{-1}\right) ^{2}}+\frac{v^{2}}{\left(
r-r^{-1}\right) ^{2}}=1
\end{equation*}%
this show that $J(\zeta )$ maps the circles $r=$constants onto the ellipses
of semi axes $r+r^{-1}$ and $\left\vert r-r^{-1}\right\vert $ and they have
common foci $\pm 2.$

In a similar way, $J(\zeta )$ maps the rays $\theta =$ constants onto
hyperbolas with the same foci $\pm 2.$

\subsection{Zero analysis}

Recall that $t_{1},t_{2},t_{3}$ are the zeros of $P(t)=0$ such that they
satisfied \ref{1.8}

In this subsecton we study the set $A$ defined as:\newline
\begin{equation}
A=\left\{ \left\vert t_{1}\left( x\right) \right\vert =\left\vert
t_{2}\left( x\right) \right\vert \right\}  \tag{2.2}  \label{2.2}
\end{equation}

\subsubsection{Structure in the $p-$plane}

To depict the set $A$ in the $x-$plane, it is better to depict the set of
points in the $p-$plane that leads to $\left\vert t_{1}\left( x\right)
\right\vert =\left\vert t_{2}\left( x\right) \right\vert $ under (\ref{2.1})

Let $p=r\exp \left( i\theta \right) ,$ $r\geq 1$ be a point in the $p-$plane

The image of $p$ in the $t-$plane is determined as follows:

\begin{eqnarray*}
q &=&J(p) \\
&=&r\exp \left( i\theta \right) +r^{-1}\exp \left( -i\theta \right) \\
&=&\left( r+r^{-1}\right) \cos \theta +i\left( r-r^{-1}\right) \sin \theta
\end{eqnarray*}

\begin{eqnarray*}
z &=&\lambda q \\
&=&\frac{i\sqrt{5}}{3}\left( r+r^{-1}\right) \cos \theta +\frac{\sqrt{5}}{3}%
\left( -r+r^{-1}\right) \sin \theta
\end{eqnarray*}

and 
\begin{eqnarray*}
t &=&z+\frac{1}{3} \\
&=&\frac{i\sqrt{5}}{3}\left( r+r^{-1}\right) \cos \theta +\frac{\sqrt{5}}{3}%
\left( -r+r^{-1}\right) \sin \theta +\frac{1}{3}
\end{eqnarray*}%
\newline
Then 
\begin{eqnarray*}
\left\vert t^{2}\right\vert &=&\left( \frac{\sqrt{5}}{3}\left(
-r+r^{-1}\right) \sin \theta +\frac{1}{3}\right) ^{2}+\left( \frac{\sqrt{5}}{%
3}\left( r+r^{-1}\right) \cos \theta \right) ^{2} \\
&=&\frac{5}{9}(r^{2}+r^{-2})+\frac{10}{9}\cos 2\theta +\frac{1}{9}+\frac{2%
\sqrt{5}}{9}\left( r^{-1}-r\right) \sin \theta
\end{eqnarray*}

Now let, $r\exp (i\theta _{0}),$ $r\exp (i\left( \theta _{0}+\frac{2\pi }{3}%
\right) )$ and $r\exp (i\left( \theta _{0}+\frac{4\pi }{3}\right) )$ be the
images of a point $x$ in the $P-$plane

For example we may assume that : $r\exp (i\theta _{0})$ leads to $t_{1}$ in
the $t-$plane and $r\exp (i\left( \theta _{0}+\frac{4\pi }{3}\right) )$
leads to $t_{2}$ .

Thus $\left\vert t_{1}^{2}\right\vert =\left\vert t_{2}^{2}\right\vert $
implies

\begin{eqnarray*}
\frac{5}{9}(r^{2}+r^{-2})+\frac{10}{9}\cos 2\theta _{0}+\frac{1}{9}+\frac{2%
\sqrt{5}}{9}\left( r^{-1}-r\right) \sin \theta _{0} &=&\frac{5}{9}%
(r^{2}+r^{-2})+ \\
&&\frac{10}{9}\cos 2\left( \theta _{0}+\frac{4\pi }{3}\right) + \\
&&\frac{1}{9}+\frac{2\sqrt{5}}{9}\left( r^{-1}-r\right) \sin \left( \theta
_{0}+\frac{4\pi }{3}\right)
\end{eqnarray*}%
\newline
Then we get after simplification: \newline
\begin{equation*}
\frac{10}{9}\left[ \cos 2\theta _{0}-\cos 2\left( \theta _{0}+\frac{4\pi }{3}%
\right) \right] =\frac{2\sqrt{5}}{9}\left( r^{-1}-r\right) \left[ -\sin
\theta _{0}+\sin \left( \theta _{0}+\frac{4\pi }{3}\right) \right]
\end{equation*}

By using the trigonometric identity 
\begin{equation*}
\cos 2a-\cos 2b=-2(\sin a-\sin b)(\sin a+\sin b)
\end{equation*}

We get:\newline
\begin{eqnarray}
\frac{2\sqrt{5}}{9}\left( r^{-1}-r\right) \left[ 2\sin \frac{2\pi }{3}\cos
\left( \theta _{0}+\frac{2\pi }{3}\right) \right] &=&-\frac{20}{9}\left[
2\sin \frac{2\pi }{3}\cos \left( \theta _{0}+\frac{2\pi }{3}\right) \right]
\times  \label{2.3} \\
&&\left[ 2\cos \frac{2\pi }{3}\sin \left( \theta _{0}+\frac{2\pi }{3}\right) %
\right]  \notag
\end{eqnarray}%
\newline
Two cases to discuss, the first one is :\newline
If $\cos \left( \theta _{0}+\frac{2\pi }{3}\right) =0$ this implies that $%
\theta _{0}=\frac{-\pi }{6}$ or $\theta _{0}=\frac{-7\pi }{6},$ then any $%
r\geq 1$ satisfies (\ref{2.3})

If $\cos \left( \theta _{0}+\frac{2\pi }{3}\neq 0\right) $ then we get after
cancelling the common factor $\cos \left( \theta _{0}+\frac{2\pi }{3}\right) 
$ leads to%
\begin{equation*}
\frac{\sqrt{5}}{9}\left( r^{-1}-r\right) =\frac{10}{9}\sin \left( \theta
_{0}+\frac{2\pi }{3}\right)
\end{equation*}

Wich is equivalent to the rectangular equation:\newline
\begin{equation*}
\left( x+\frac{\sqrt{15}}{2}\right) ^{2}+\left( y-\frac{\sqrt{5}}{2}\right)
^{2}=6
\end{equation*}

we study this result as a lemma

\begin{lemma}
The condition for $\left\vert t_{1}\right\vert =\left\vert t_{2}\right\vert $
in the $p-$plane where $re^{i\theta _{0}}=p_{1}$ corresponds to $t_{1}$ and $%
re^{i\left( \theta _{0}+\frac{4\pi }{3}\right) }=p_{2}$ corresponds to $%
t_{2} $ is:\newline
1- $p_{1}=re^{-i\frac{\pi }{6}}$ or $re^{-i\frac{7\pi }{6}},$ $r\geq 1$ or%
\newline
2- $p_{1}$ lies in the circular arc of the circle 
\begin{equation*}
\left( x+\frac{\sqrt{15}}{2}\right) ^{2}+\left( y-\frac{\sqrt{5}}{2}\right)
^{2}=6
\end{equation*}%
in the region $r\geq 1.$
\end{lemma}

We need to impose $\left\vert t_{1}\right\vert \leq \left\vert
t_{3}\right\vert $ to make sure that $\left\vert t_{1}\right\vert \leq
\left\vert t_{2}\right\vert \leq \left\vert t_{3}\right\vert $ \newline
Recall that $t_{3}$ corresponds to $re^{i\left( \theta _{0}+\frac{2\pi }{3}%
\right) \text{ }}$ in the $p-$plane\newline
Since $\left\vert t_{1}\right\vert <\left\vert t_{3}\right\vert $ this
implies $\left\vert t_{1}\right\vert ^{2}<\left\vert t_{3}\right\vert ^{2}$
then we have:\newline
\begin{eqnarray*}
\frac{2\sqrt{5}}{9}\left( r^{-1}-r\right) \left[ 2\sin \left( \frac{-\pi }{3}%
\right) \cos \left( \theta _{0}+\frac{\pi }{3}\right) \right] &>&\frac{-20}{9%
}\left[ 2\sin \left( \frac{-\pi }{3}\right) \cos \left( \theta _{0}+\frac{%
\pi }{3}\right) \right] \times \\
&&\left[ 2\sin \left( \theta _{0}+\frac{\pi }{3}\right) \cos \left( \frac{%
-\pi }{3}\right) \right]
\end{eqnarray*}%
\newline
To solve this inequality two cases to discuss:

1- If $\cos \left( \theta _{0}+\frac{\pi }{3}\right) >0$ wich is equivalent
to 
\begin{equation}
\frac{-5\pi }{6}<\theta _{0}<\frac{\pi }{6}  \tag{2.4}  \label{2.4}
\end{equation}%
\newline
Then after cancelling $\cos \left( \theta _{0}+\frac{\pi }{3}\right) $ we
get: 
\begin{equation*}
\frac{2\sqrt{5}}{9}\left( r^{-1}-r\right) <\frac{-20}{9}\sin \left( \theta
_{0}+\frac{\pi }{3}\right)
\end{equation*}%
\newline
or in rectangular form%
\begin{equation}
\left( x+\frac{\sqrt{15}}{2}\right) ^{2}+\left( y-\frac{\sqrt{5}}{2}\right)
^{2}>6  \tag{2.5}  \label{2.5}
\end{equation}

The corresponding region in this case follows after combining (\ref{2.4})
and (\ref{2.5}). Namely%
\begin{equation}
\left( x+\frac{\sqrt{15}}{2}\right) ^{2}+\left( y-\frac{\sqrt{5}}{2}\right)
^{2}>6,\text{ }r\geq 1,\text{ and }\frac{-5\pi }{6}<\theta _{0}<\frac{\pi }{6%
}  \tag{2.6}  \label{2.6}
\end{equation}%
\newline
2- If: $\cos \left( \theta _{0}+\frac{2\pi }{3}\right) <0$ this is
equivalent to: 
\begin{equation}
\frac{\pi }{6}<\theta _{0}<\frac{7\pi }{6}  \tag{2.7}  \label{2.7}
\end{equation}%
\newline
And the inequality to solve is: 
\begin{equation*}
\frac{2\sqrt{5}}{9}\left( r^{-1}-r\right) >\frac{-10}{9}\sin \left( \theta
_{0}+\frac{\pi }{3}\right)
\end{equation*}

In a similar way we get the region for the case 2:

\begin{equation}
\left( x+\frac{\sqrt{15}}{2}\right) ^{2}+\left( y-\frac{\sqrt{5}}{2}\right)
^{2}<6,\text{ }r\geq 1,\text{ and }\frac{\pi }{6}<\theta _{0}<\frac{7\pi }{6}
\tag{2.8}  \label{2.8}
\end{equation}%
\newline
Let region $\Gamma _{1}$ correspond to (\ref{2.6}) and region $\Gamma _{2}$
correspond to (\ref{2.8})

Then: 
\begin{equation*}
\Gamma _{1}\cup \Gamma _{2}=\left\{ \text{points in }p-\text{plane that
corresponds to }\left\vert t_{1}\right\vert <\left\vert t_{3}\right\vert
\right\}
\end{equation*}

Now, imposing the condition stated in lemma 4, we get the point set $L_{1}$
of the point $p_{1}$ that corresponds to root $t_{1}$.\newline
Explicitly, $L_{1}$ is the point set defined as: \newline
$L_{1}=\left\{ \Gamma _{1}\cup \Gamma _{2}\right\} \cap \left\{ \left\{ 
\text{ray}:\theta =\frac{-\pi }{6}\right\} \cup \left\{ \left( x+\frac{\sqrt{%
15}}{2}\right) ^{2}+\left( y-\frac{\sqrt{5}}{2}\right) ^{2}=6\right\}
\right\} $\newline
Similarly, we defined $L_{2}$ as the rotation of $L_{1}$ through an angle of 
$\frac{4\pi }{3}$ \newline
and $L_{3}$ through $\frac{2\pi }{3}$.\newline
It is clear that $L_{2}$ consists of points in the $p-$plane that
corresponds to $t_{2\text{ }}$ and $L_{3}$ corresponds to $t_{3}.$ Thus we
have proved the following:

\begin{theorem}
The condition for the points in the $p-$plane for which $\left\vert
t_{1}\right\vert =\left\vert t_{2}\right\vert <\left\vert t_{3}\right\vert $
is $p_{i}\in L_{i\text{ }}$, for $i=1,2,3$\newline
\end{theorem}

\textbf{The pull back action}

\begin{theorem}
The curve represented by set $A$ in the $x-$plane is the curve which is
obtained by pulling the symmetric curve in the $p-$plane through the
conformal mappings back to the $x-$plane.
\end{theorem}

Let $Z\left( \tau _{n}\right) $ denoted the zero attractor of $\tau
_{n}\left( x\right) $

To see the zero attractor of $T_{n}\left( x,y\right) $, by using (\ref{1.9})
we have $\tau _{n}(x)=x^{\left( \frac{1}{2}-n\right) }T_{n}\left( x,y\right) 
$

Assume $x_{0}\in Z\left( \tau _{n}\right) $ again by equation (\ref{1.9}) $%
\tau _{n}(x_{0})=x_{0}^{\frac{1}{2}-n}T_{n}\left( x_{0},y_{0}\right) $ where 
$y_{0}=x_{0}^{\frac{3}{2}}$. Since $\tau _{n}(x_{0})=0$, we get $T_{n}\left(
x_{0},y_{0}\right) =0$. Hence $x_{0}^{\frac{1}{2}-n}\in Z\left( T_{n}\right) 
$.

\section{\textbf{Zero attractor}}

The concept of attractors was introduced in 1965 by Auslander et al (\cite%
{ABS}),

According to Milnor (\cite{M}) the attractors have played an increasingly
important role in thinking about dynamical systems.

\begin{definition}
Let $\left\{ q_{n}(x)\right\} _{n\geq 0}$ be a sequence of polynomials,
where the degree of $q_{n}(x)$ increases to infinity as $n\longrightarrow
\infty $. \newline
A set $A$ in the $x-$plane is called the asymptotic zero attractor of zeros
of $\left\{ q_{n}(x)\right\} _{n\geq 0}$ if the following two conditions
holds:\newline
$\left( 1\right) :$ Let $A_{\varepsilon }=\cup _{x\in A}B(x,\varepsilon )$
where $B(x,\varepsilon )$ is the open disc centred at $x$ with radius $%
\varepsilon $, $A_{\varepsilon }$ is just a neighborhood of $A$, $\exists
n_{0}(\varepsilon )$ such that $\forall n\geq n_{0}$ all the zeros of $%
q_{n}(x)$ are in $A_{\varepsilon }.$\newline
$\left( 2\right) :$ For all $x\in A,\forall \varepsilon >0,\exists n_{1}\in 
\mathbb{N}
$ such that $n_{1}(x,\varepsilon )$ and there exist a zero $r$ of $\left\{
q_{n}(x)\right\} _{n\geq 0}$ such that $r\in B(x,\varepsilon ).$
\end{definition}

The condition for $x$ values for which $\left\vert t_{1}\right\vert
=\left\vert t_{2}\right\vert $ is a curve. Let $L$ be this curve.\newline
For all $\varepsilon >0,$ let $L_{\varepsilon }$ be the $\varepsilon -$%
neighborhood of $L$ in the $x-$plane.\newline
In this section we will give a justification that, for all large $n$, all
the zeros of $\tau _{n}(x)$ are contained in $L_{\varepsilon }$.

According to (\ref{1.11}), the asymptotics of $\tau _{n}(x)$ depends on the
magnitudes of the zeros of $P\left( t\right) =0$.

Let $B=\left\{ x\in x-\text{plan}:\left\vert t_{1}\left( x\right)
\right\vert <\left\vert t_{2}\left( x\right) \right\vert \right\} $\newline
Obviously, $B$ is an open region in the $x-$plane.

Then we have

\begin{lemma}
There exists a non negative real number $\rho $ such that for all large $n$,
the zeros of $\tau _{n}(x)$ are contained in the disc $D_{\rho }=\left\{
x:\left\vert x\right\vert \leq \rho \right\} $
\end{lemma}

\begin{proof}
The point infinity point in the extended $x-$plane is mapped to $0$ under
the mapping $\frac{1}{x}$ \newline
By the sequence of mappings defined in (\ref{2.1}) $0$ is mapped to $\beta =%
\frac{-11\sqrt{5}i}{25}$ corresponds to $t_{1}=0$, $t_{2}=\frac{-1-i\sqrt{5}%
}{3}$ and $t_{3}=\frac{-1+i\sqrt{5}}{3}$\newline
The choice of $t_{2}$ and $t_{3}$ is arbitrary since they have the same
magnitude.\newline
In this situation we have: 
\begin{eqnarray*}
\left\vert t_{1}\right\vert &=&0 \\
\left\vert t_{2}\right\vert &=&\frac{2}{3} \\
\left\vert t_{3}\right\vert &=&\frac{2}{3}
\end{eqnarray*}

and 
\begin{eqnarray*}
\left\vert P^{^{\prime }}(t_{1})\right\vert &=&2 \\
\left\vert P^{^{\prime }}(t_{2})\right\vert &=&\frac{14+4\sqrt{5}}{3} \\
\left\vert P^{^{\prime }}(t_{3})\right\vert &=&\frac{14-4\sqrt{5}}{3}
\end{eqnarray*}

Hence $x=\infty \in B$. \newline
Since $B$ is open region, there exists a non negative real number $\rho >0$, 
$\left\{ x:\left\vert x\right\vert \geq \rho \right\} \subseteq B$ such that
for all $x\in \left\{ x:\left\vert x\right\vert \geq \rho \right\} $, we
have : $\left\{ 
\begin{array}{c}
\left\vert t_{1}\right\vert =0 \\ 
\left\vert t_{2}\right\vert >\frac{1}{2} \\ 
\left\vert t_{3}\right\vert >\frac{1}{2}%
\end{array}%
\right. $\newline
\begin{equation*}
P^{^{\prime }}(t_{1})=2
\end{equation*}%
\begin{equation*}
1\leq \left\vert P^{^{\prime }}(t_{2})\right\vert
\end{equation*}%
and 
\begin{equation*}
1\leq \left\vert P^{^{\prime }}(t_{3})\right\vert
\end{equation*}%
\newline
Now from (\ref{1.11}) we get \newline
\begin{equation*}
\tau _{n}(x)=\frac{t_{2}^{1-n}}{P^{^{\prime }}(t_{2})}\left[ 1+\frac{%
P^{^{\prime }}(t_{2})}{P^{^{\prime }}(t_{1})}\left( \frac{t_{1}}{t_{2}}%
\right) ^{n-1}+\frac{P^{^{\prime }}(t_{2})}{P^{^{\prime }}(t_{3})}\left( 
\frac{t_{2}}{t_{3}}\right) ^{n-1}\right]
\end{equation*}%
\newline
This gives the estimate%
\begin{eqnarray*}
\left\vert \tau _{n}(x)\right\vert &=&\left\vert \frac{t_{2}^{1-n}}{%
P^{^{\prime }}(t_{2})}\right\vert \left\vert 1+\frac{P^{^{\prime }}(t_{2})}{%
P^{^{\prime }}(t_{1})}\left( \frac{t_{1}}{t_{2}}\right) ^{n-1}+\frac{%
P^{^{\prime }}(t_{2})}{P^{^{\prime }}(t_{3})}\left( \frac{t_{2}}{t_{3}}%
\right) ^{n-1}\right\vert \\
&=&\left( \frac{3}{2}\right) ^{1-n}\left\vert 1-\left( \frac{2}{3}\right)
^{n-1}\right\vert \\
&=&\left( \frac{3}{2}\right) ^{1-n}\left\vert 1-\left( \frac{2}{3}\right)
^{n-1}\right\vert \\
&\geq &\frac{1}{2}\frac{3^{n}}{2}>0
\end{eqnarray*}%
\newline
for all large $n.$ This completes the proof of the lemma.
\end{proof}

We know that the zeros of $\tau _{n}(x)$ are contained in the disk $\left\{
x:\left\vert x\right\vert \leq \rho \right\} $, then the next lemma allow us
to know where the zeros are going.

\begin{lemma}
Let $K$ be a compact subset of $B$\newline
Then $\tau _{n}(x)t_{2}^{n-1}P^{^{\prime }}(t_{2})\rightarrow 1$ uniformly
for all $x\in $ $K$ as $n\longrightarrow \infty $
\end{lemma}

\begin{proof}
By lemma 3 , if $x\in B$ , then $P(t)$ does not have any repeated zeros so
that $P^{^{\prime }}(t_{i})\neq 0$ for $i=1,2,3,$ by assumption $K$ is
compact, there must exist a $\lambda >0$ and a number $M$ such that $M\geq
\left\vert P^{^{\prime }}(t_{i})\right\vert \geq \lambda >0$ uniformly for $%
x\in K$ and $i=1,2,3.$

Again from (\ref{1.11}) we have

\begin{equation}
\tau _{n}(x)t_{2}^{n-1}P^{^{\prime }}(t_{2})=\left\vert 1+O\left( \left\vert 
\frac{t_{2}}{t_{3}}\right\vert ^{n-1}\right) \right\vert  \tag{3.1}
\label{3.1}
\end{equation}

The big $O$ terms approach zero uniformly because $K$ is a compact set.
Hence the result follows.
\end{proof}

\begin{corollary}
The region $B$ contains no points of $Z\left( \tau _{n}\right) $.
\end{corollary}

\begin{proof}
By lemma 5, $t_{2}^{n-1}P^{^{\prime }}(t_{2})$ is never zero for $x\in K$, $%
\tau _{n}(x)$ and $\tau _{n}(x)t_{2}^{n-1}P^{^{\prime }}(t_{2})$ have the
same zero set. Let $x_{0}$ be an arbitrary point in $Z\left( \tau
_{n}\right) $. Then $x_{0}$ is an accumulation point of zeros of $\tau
_{n}(x)$. Suppose $x_{0}\in B$, since $B$ is open there exists a
neighborhood $N(x_{0})$ of $x_{0}$ such that $\overline{N(x_{0})}\subseteq B$%
. with respect to this neighborhood , there exists an infinite sequence of
integers $n_{j}$, $j\geq 1$ and a zero $r_{n_{j}}$ of $\tau _{n_{j}}(x)$
such that $r_{n_{j}}\in N(x_{0})$ for all $j\geq 1$. So for all sufficiently
large $j$ we have $\tau _{n_{j}}(r_{n_{j}})=0$, but this violates the
asymptotic estimate in (\ref{3.1}) with $K$ chosen as $\overline{N(x_{0})}$,
which is a contradiction, hence $x_{0}\notin B$ . Hence the proof of
corollary is completed.
\end{proof}

\begin{corollary}
For all large $n$, all zeros of $\tau _{n}(x)$ are contained in $%
L_{\varepsilon }$, that is $Z\left( \tau _{n}\right) \subset L$.
\end{corollary}


\begin{thebibliography}{99}
\bibitem{ABB} Ait-Amrane L, Belbachir H, Betina K.: Periods of Morgan-Voyce
and elliptic curves Math. Slovaca 66 (2016), No. 6, 1267--1284.

\bibitem{ABS} Auslander, J., Bhatia, N.P., Seibert, P.: Attractors in
dynamical systems. Bol. Soc. Mat. Mex. 9, 55-66 (1964).

\bibitem{BB} Belbachir, H. Bencherif, F: Linear recurrent sequences and
powers of a square matrix. Electronic journal of combinatorial number
theory, Integers 6\newline
(2006), A12, 17 pp.

\bibitem{BKS} Belbachir, H. Komatsu, T. Szalay, L.: Characterization of
linear recurrences associated to rays in Pascal's triangle. In: Diophantine
analysis and related fields 2010. AIP Conf. Proc., 1264, Amer. Inst. Phys.,
Melville, NY, 2010, pp. 90--99.

\bibitem{BKS1} Belbachir, H. Komatsu, T. Szalay, L.: Linear recurrences
associated to rays in Pascal's triangle and combinatorial identities Math.
Slovaca 64 (2014), 287--300.

\bibitem{BG1} Boyer, R., Goh, W.: On the zero attractor of the Euler
polynomials. Adv. Appl. mahs.38, 97-132 (2007)

\bibitem{BG2} Boyer, R., Goh, W.: On the zero attractor of the partition
polynomials. http://arxiv.org/abs/0809.1266

\bibitem{HRS} He, M.X., Ricci, P.E., Simon, D.: Numerical results on the
zeros of generalized Fibonacci polynomials. calcolo 34, 25-40 (1997)

\bibitem{GHR} Goh, W \textperiodcentered\ He, M X \textperiodcentered\ Ricci
P E.: On the universal zero attractor\newline
of the Tribonacci-related polynomials. Calcolo 46, 95--129 (2009)

\bibitem{K} Koshy, T.: Fibonacci and Lucas Numbers with applications (Wiley,
New york, 2001)

\bibitem{M} Milnor, J.: On the Concept of Attractor. Commun. Math. Phys.
99,177-195 (1985)

\bibitem{N} Nehari, Z.: Conformal mappings (McGraw-Hill, New York, 1952)

\bibitem{S} Stanley, R.: http://www-math.mit.edc/-rstan/zeros

\bibitem{V} Vince, A.: Period of a linear recurrence, Acta Arith. 39 (1981),
303--311.
\end{thebibliography}
\end{document}